\documentclass[journal]{IEEEtran}
\usepackage{pgfplots}
\usepackage{amsthm}
\usepackage{dblfloatfix} 

\usepackage{microtype}
\usepackage{units}
\usepackage{amsmath}
\usepackage{amssymb}
\usepackage{mathtools}
\usepackage{color}
\usepackage{url}
\usepackage{multicol}
\usepackage{hyperref}
\usepackage{graphicx}
\usepackage{caption}
\usepackage{subcaption}
\usepackage{cite}
\usepackage{booktabs}
\usepackage{pgfplots}\pgfplotsset{compat=1.4}
\usepackage{subdepth}
\usepackage{algorithmic}
\usepackage{algorithm}
\newcommand{\bq}{\begin{eqnarray}}
\newcommand{\eq}{\end{eqnarray}}

\renewcommand{\Pi}{\mbox{\LARGE$\pi$}}
\newcommand{\RR}{\mathbb{R}}

\newtheorem{lemma}{Lemma}
\newtheorem{corollary}{Corollary}
\newtheorem{theorem}{Theorem}

\graphicspath{{./figs/}}

\title{Efficient quadratic penalization through the partial minimization technique}

 \author{%
    Aleksandr Y. Aravkin$^\dagger$\thanks{$^\dagger$Department of Applied Mathematics, University of Washington, Seattle, WA
(saravkin@uw.edu)}
    \and Dmitriy Drusvyatskiy$^*$\thanks{$^*$Department of Mathematics, University of Washington, Seattle, WA (ddrusv@uw.edu)}
    \and Tristan van Leeuwen$^{**}$
    \thanks{$^{**}$Department of Mathematics, Utrecht University, Utrecht, Nethelands (T.vanLeeuwen@uu.nl)}}

\begin{document}

\maketitle

\begin{abstract}
Common computational problems, such as parameter estimation in dynamic models and PDE constrained optimization, 
 require data fitting over a set of auxiliary parameters subject to physical constraints over 
an underlying state.  Naive quadratically penalized formulations, commonly used in practice, suffer from inherent ill-conditioning. %, scaling with the penalty parameter.  
We show that surprisingly the partial minimization technique regularizes the problem, making it well-conditioned. 
 This viewpoint sheds new light on variable projection techniques, as well as the penalty method for PDE constrained optimization, 
 and motivates robust extensions. 
In addition, we outline an inexact analysis, 
showing that the partial minimization  subproblem can be solved very loosely in each iteration. 
We illustrate the theory and algorithms on boundary control, optimal transport, and parameter estimation 
for robust dynamic inference. 
 \end{abstract}

%\section*{Plan}
%
%\begin{enumerate}
%\item Introduction, specify class of problems  
%\item Present differentiability theorem, and discuss corollaries for specific problem class. 
%%\item Can we say something about inexactness, and its effects?
%\item Machine learning survey  
%\item For experiments, compare with FULL solution, using Overton's code on both full and reduced objectives. (Tristan has already done). 
%\item MKL: new approach
%\item conclusions 
%\end{enumerate}

\section{Introduction}

In this work, we consider a structured class of optimization problems having the form 
\begin{equation}\label{eqn:main_stuff}
\min_{y,u}~ f(y)+g(u) \qquad \mbox{subject to} \qquad  A(u)y = q.
\end{equation}
Here, $f\colon\RR^n\to\RR$ is convex and smooth, 
$q\in\RR^n$ is a fixed vector, and $A(\cdot)$ is a smoothly varying invertible matrix. 
For now, we make no assumptions on $g\colon\RR^d\to\RR$, though in practice, it is typically either a smooth 
 or a `simple' nonsmooth function.   
%Large scale 
Optimization problems of this form often appear in PDE constrained optimization \cite{virieux2009overview,tarantola2005,AravkinFHV:2012}, Kalman filtering \cite{kalman,anderson1979optimal,aravkin2014robust}, boundary control problems \cite{Gugat2009,Philip2009}, and optimal transport \cite{Haker2004,Ambrosio2013}. Typically, $u$ encodes auxiliary variables while  $y$ encodes the state of the system; the constraint $A(u)y=q$ corresponds to a discretized PDE describing the physics. 

\begin{figure}[!h]
	\centering
				\includegraphics[scale=.4]{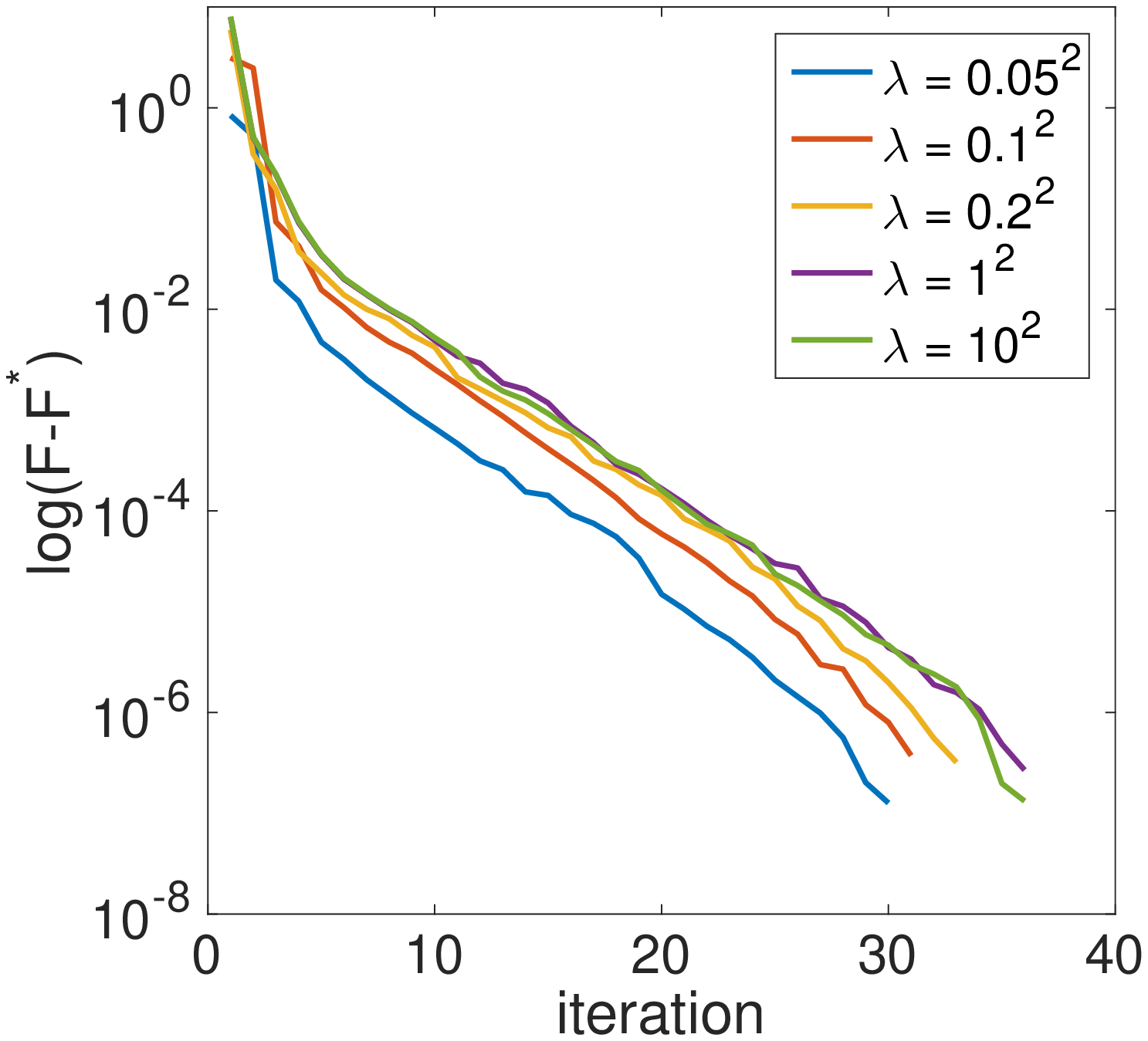}\\
				\includegraphics[scale=.4]{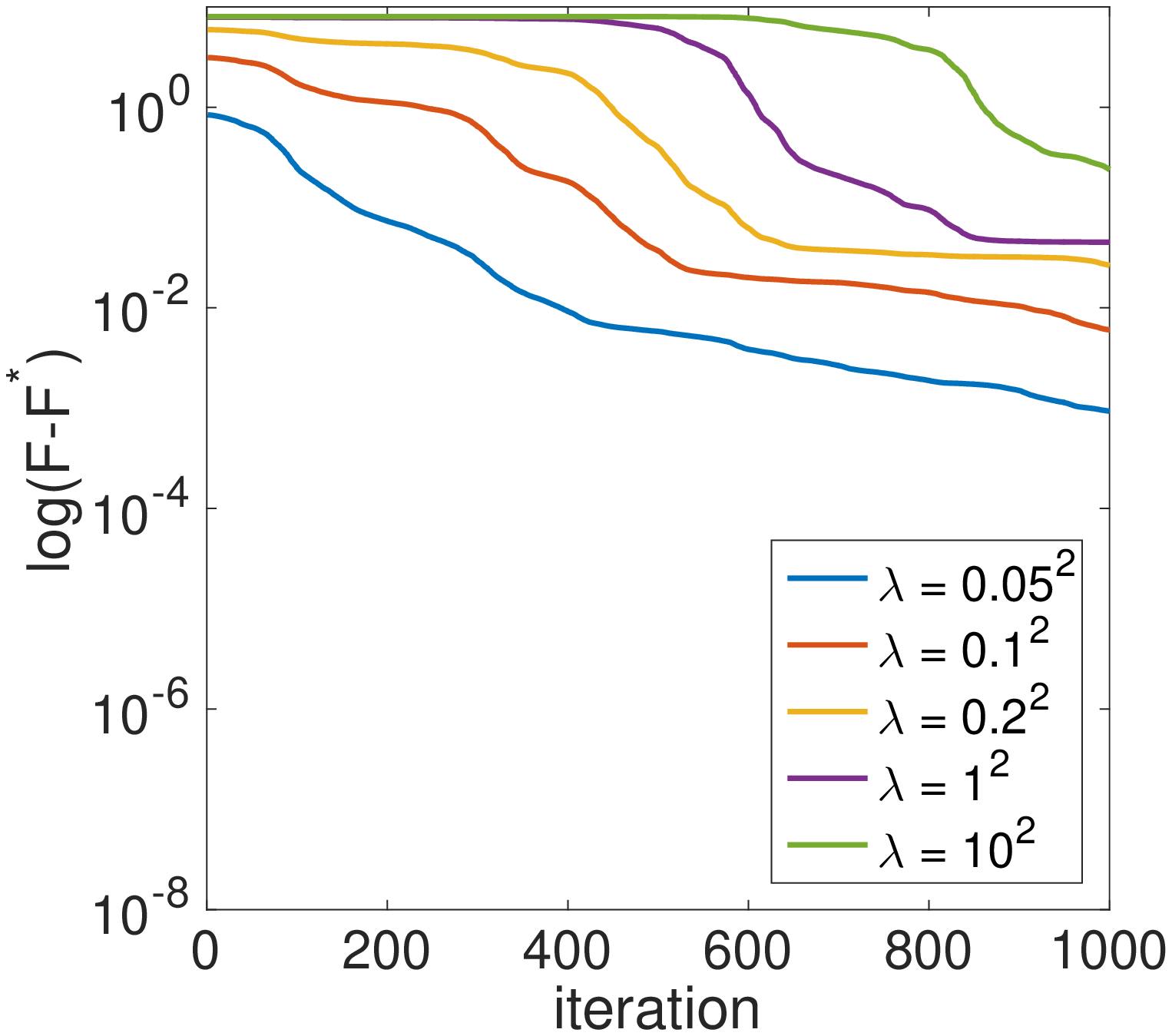}

%			\vspace{-2cm}
	\caption{   
	Top panel shows $\log(F_k - F^*)$ of L-BFGS {\bf with} partial minimization applied to~\eqref{eqn:uncon_form}  
	for the boundary control problem in Section~\ref{sec:num_ill}. Each partial minimization solves a least squares problem 
	in $y$. 
	Bottom panel shows $\log(F_k - F^*)$  of L-BFGS {\bf without} partial minimization applied to~\eqref{eqn:uncon_form}, 
	for the same values of $\lambda$.  Both methods are initialized at random $u$, and $y$ that minimizes $F(\cdot, u)$).
	Performance of L-BFGS without partial minimization degrades as $\lambda$ increases, while performance of 
	L-BFGS with partial minimization is insensitive to $\lambda$. }
	\vspace{-.3cm}
	\label{fig:boundaryComp}
\end{figure}

Since the discretization $A(u)y=q$ is already inexact, it is appealing to relax it in the formulation. 
%\textcolor{red}{Tristan, can you please add references for boundary control and optimal transport?}
A seemingly naive relaxation approach is based on the quadratic penalty:
\begin{equation}\label{eqn:uncon_form}
\min_{y,u}~ F(y,u) :=  f(y)+g(u) +\lambda\cdot \|A(u)y - q\|^2.
\end{equation}
Here $\lambda>0$ is a relaxation parameter for the equality constraints in \eqref{eqn:main_stuff}, corresponding to relaxed physics. The classical quadratic penalty method in nonlinear programming proceeds by applying an iterative optimization algorithm to the unconstrained problem \eqref{eqn:uncon_form} until some termination criterion is satisfied, then increasing $\lambda$, and repeating the procedure with the previous iterate used as a warm start. For a detailed discussion, see e.g. \cite[Section 17.1]{Noc_Wright}.
The authors of \cite{van2015penalty} observe that this strategy helps to avoid extraneous local minima, in contrast to the original formulation \eqref{eqn:main_stuff}. From this consideration alone, the formulation \eqref{eqn:uncon_form} appears to be useful.

Conventional wisdom teaches us that the quadratic penalty technique is rarely appropriate. The difficulty is that one must allow $\lambda$ to tend to infinity in order to force near-feasibility in the original problem \eqref{eqn:main_stuff}; the residual error
$\|A(u)y - q\|$ at an optimal pair $(u,y)$ for the problem $\eqref{eqn:uncon_form}$ is at best on the order of $\mathcal{O}(1/\lambda)$. Consequently, the maximal eigenvalue of the Hessian of the penalty term scales linearly with $\lambda$ and the problems \eqref{eqn:uncon_form} become computationally difficult. Indeed, 
the maximal eigenvalue of the Hessian determines the behavior of numerical methods (gradient descent, quasi-Newton) far away from the solution -- a regime in which most huge scale problems are solved. Figure~\ref{fig:boundaryComp} is a simple numerical illustration of this inherent difficulty on a boundary control problem; see Section~\ref{sec:num_ill} for more details on the problem formulation. The bottom panel in the figure tracks progress of the objective function in~\eqref{eqn:uncon_form} when an L-BFGS method is applied jointly in the variables $(u,y)$. 
After 1000 iterations of L-BFGS, the objective value significantly increases with increasing $\lambda$, while the actual minimal value of
the objective function converges to that of~\eqref{eqn:main_stuff}, and so hardly changes. In other words, performance of the method scales poorly with $\lambda$, illustrating the ill-conditioning. 

In this paper, we show that by using a simple {\em partial minimization} step, this complexity blow-up can be avoided entirely. The resulting algorithm is perfectly suited for many large-scale problems, where satisfying 
the constraint $A(u)y =  q$ to high accuracy is not required (or even possible). 
%The idea is elementary. 
The strategy is straightforward: we rewrite~\eqref{eqn:uncon_form} as 
\begin{equation}\label{eqn:reduced}
\min_{u}~\widetilde \varphi(u) + g(u),
\end{equation}
where the function 
$\widetilde \varphi(u)$ is defined implicitly by 
\begin{equation}\label{eqn:subprob_eff}
\widetilde \varphi(u)=\min_y\, \left\{f(y)  + \lambda\cdot \|A(u)y - q\|^2\right\}.
\end{equation}
We will call $\widetilde \varphi(\cdot)$ the {\em reduced function}.
Though this approach of {minimizing out} the variable $y$, sometimes called {\it variable projection}, is widely used 
(e.g. \cite{Golub2003,Aravkin2012c,van2015penalty}), little theoretical justification for its superiority is known. In this work, we show that not only does partial minimization perform well numerically for the problem class \eqref{eqn:uncon_form}, but is also theoretically grounded. We prove that surprisingly the Lipschitz constant of $\nabla\widetilde\varphi$ is bounded by a constant independent of $\lambda$. Therefore, iterative methods can be applied directly  to the formulation \eqref{eqn:reduced}. 
The performance of the new method is illustrated using a toy example (top panel of Figure~\ref{fig:boundaryComp}). We use L-BFGS to attack the outer problem; solving it within 35 iterations. The inner solver for the toy example 
simply solves the least squares problem in $y$.

The inner problem~\eqref{eqn:subprob_eff} can be solved efficiently since its condition number is nearly independent of $\lambda$. When $f$ is a convex quadratic and $A(u)$ is sparse, one can apply  sparse direct solvers or iterative methods such as  
LSQR~\cite{paige1982lsqr}. More generally, when $f$ is an arbitrary smooth convex function, one can apply first-order methods, which converge globally linearly with the rate governed by the condition number of the strongly convex objective in \eqref{eqn:subprob_eff}.  Quasi-newton methods or variants of Newton's method are also available; even if $g(u)$ is nonsmooth, BFGS methods can still be applied.

The outline of the paper is as follows.
In Section~\ref{sec:theory}, we present complexity guarantees of the partial minimization technique.
%In Section~\ref{sec:robust}, we discuss robust inverse problem formulations, and describe an efficient implementation of the variably projected penalty method. 
In  Section~\ref{sec:num_ill}, we  numerically illustrate  the overall approach  on boundary control and optimal transport problems, and on tuning an oscillator  from very noisy measurements.

%\textcolor{red}{when quadratic discuss sparse exact methods (sasha)}

%\section{Nonsmooth Variable Projection in Applications \label{sec:survey}}
%\input{survey.tex}

\section{Theory\label{sec:theory}}

%We first explain the issues with a naive optimization approach for~\eqref{eqn:uncon_form}.
%In this approach, 
In this section, we show that the proposed framework is insensitive to the parameter $\lambda$. Throughout we assume that $f$ and $A(\cdot)$ are $C^2$-smooth, $f$ is convex, and $A(u)$ is invertible for every $u\in\RR^n$. % and that $f$ is in addition convex. 
%We make no assumptions on $g$ for now. 

%\subsection{Conditioning is independent of $\lambda$}
To shorten the formulas, in this section, we will use the symbol $A_u$ instead of $A(u)$ throughout. 
Setting the stage, define the function 
\[
\varphi(u,y) = f(y) + \frac{\lambda}{2} \|A_u y-q\|^2.
\]
%Then the objective function of the penalized subproblem \eqref{eqn:uncon_form} is simply
%\[
%\varphi(u,y) = g(u) + P(y,u).
%\]
A quick computation shows
\begin{equation}
\label{eq:fullGradient}
\begin{aligned}
\nabla_y \varphi(u,y) & = \nabla f(y) + \lambda A_u^T(A_u y - q),\\
\nabla_u \varphi(u,y) &=  \lambda G(u,y)^T(A_u y - q), 
\end{aligned}
\end{equation}
where $G(u,y)$ is the Jacobian with respect to $u$ of the map $u\mapsto A_u y$. Clearly the Lipschitz constant ${\rm Lip}(\nabla \varphi)$ scales with $\lambda$. 
This can be detrimental to numerical methods. For example, basic gradient descent will find a point $(u_k, y_k)$ satisfying $\|\nabla  \varphi(u_k,y_k)\|^2<\epsilon$ after at most $\mathcal{O}\left(\frac{{\rm Lip}(\nabla \varphi)(\varphi(u_0,x_0)-\varphi^*)}{\epsilon}\right)$ iterations \cite[Section 1.2.3]{nest_lect_intro}.
%A priori, we have no information about the general behavior of $A(u)$ and $G(u,y)$, 
%and clearly the Lipschitz constant of the gradient of $\varphi$ is proportional to $\lambda$. Since in 
%practice we want to look at high lambda, this destroys the convergence rate of any first-order algorithm. 
%
%\textcolor{red}{Dima, can you give an example of a complexity bound and explain what it means?}

As discussed in the introduction, minimizing $\varphi$ amounts to the minimization problem $\min_u \widetilde \varphi(u)$ for the reduced function $\widetilde \varphi$ defined in 
 \eqref{eqn:subprob_eff}. Note since $f$ is convex and $A_u$ is invertible, the function $\varphi(u,\cdot)$ 
admits a unique minimizer, which we denote by $y_u$.  Appealing to the classical implicit function theorem (e.g. \cite[Theorem~10.58]{RTRW}) we deduce that $\widetilde \varphi$ is differentiable with 
\[
\nabla \widetilde \varphi(u) = \nabla_u \varphi(\cdot, y_u)\Big|_{u} = \lambda G(u,y_u)^T (A_u y_u - q).
\]
We aim to upper bound the Lipschitz constant of $\nabla \widetilde\varphi$ by a quantity independent of $\lambda$. We start by estimating the residual $\|A_u y_u-q\|$. 

Throughout the paper, we use the following simple identity. Given an invertible map $F\colon\RR^n\to\RR^n$ 
and invertible matrix $C$, for any points $x\in \RR^n$ and  nonzero  $\lambda\in\RR$, we have
\begin{enumerate}
\item $F^{-1}(\lambda x)=(\lambda^{-1}F)^{-1}(x),$ and
\item $C\circ F^{-1} \circ C^T = \left(C^{-T}\circ F \circ C^{-1}\right)^{-1}$.
\end{enumerate}
We often apply this observation to the invertible map $F(x) = \nabla f(x) + Bx$, where $B$ is a positive definite matrix.

\begin{lemma}[\label{lem:resBound}Residual bound] For any point $u$, the inequality holds:
\[
\|A_uy_u - q\| \leq \frac{\|\nabla (f\circ A^{-1}_u)(q)\|}{\lambda}.
\]
\end{lemma}
\begin{proof}
	Note that the minimizers $y_u$ of $\varphi(u,\cdot)$ are characterized by first order optimality conditions 
\begin{equation}\label{eqn:yucar}
0=\nabla f(y)+\lambda\cdot A_u^T(A_u y-q).
\end{equation}
Applying the implicit function theorem, we deduce that $y_u$ depends $C^2$-smoothly on $u$ with $\nabla_u y_u$
given by 
\begin{equation*}
%\label{eq:Lipyu}
 -\left(\nabla^2 f(y_u) + \lambda A_u^TA_u\right)^{-1}\nabla_u \left( \lambda A(\cdot)^T
\left(A(\cdot) y_u - q\right)\right)(u).
\end{equation*}
On the other hand, from the equality \eqref{eqn:yucar} we have
\begin{equation*}
\label{eq:yu}
\begin{aligned}
y_u & =(\nabla f+\lambda A_u^TA_u)^{-1}(\lambda A_u^Tq)
 = \left(\frac{\nabla f}{\lambda}+A_u^TA_u\right)^{-1}A_u^Tq.
%= \frac{1}{\lambda}A(u)^{-1}\left(I + \frac{1}{\lambda}A(u)^{-T} \nabla f A(u)^{-1}\right)^{-1}q
\end{aligned}
\end{equation*}
Therefore, we deduce
\begin{equation*}
\begin{aligned}%\label{eqn:main_res}
A_u y_u - q &= A_u\left(\frac{1}{\lambda} \nabla f + A_u^TA_u\right)^{-1}A_u^Tq -q\\
  &= \left(\left(\frac{1}{\lambda} A^{-T}_u \circ \nabla f \circ A^{-1}_u + I\right)^{-1} - I\right)q.
\end{aligned}
\end{equation*}
Define now the operator 
$$F:=\frac{1}{\lambda} A^{-T}_u \circ \nabla f \circ A^{-1}_u + I$$
and the point $z := F(q)$. Note that $$F(x)-x=\frac{1}{\lambda}\nabla (f\circ A^{-1}_u)(x).$$
Letting  $L$ be a Lipschitz constant of $F^{-1}$, we obtain 
\[
\begin{aligned}
\|A_uy_u - q\|&=\|F^{-1}(q)-F^{-1}(z)\| \leq L \|q-z\| \\
&= L \|q - F(q)\|  = \frac{L}{\lambda} \|A^{-T}_u\nabla f(A^{-1}_uq)\|.
\end{aligned}
\]
Now the inverse function theorem yields for any point $y$ the inequality
\[
\begin{aligned}
\|\nabla F^{-1}(y)\|&= \|\nabla F(F^{-1}(y))^{-1}\| \\
&=\left\|\left(\frac{1}{\lambda} \nabla^2 (f\circ A_u^{-1})(F^{-1}(y))+I\right)^{-1}\right\| \leq 1, 
\end{aligned}
\]
where the last inequality follows from the fact that by convexity of $f\circ A_u^{-1}$ all eigenvalues of $\nabla^2 (f\circ A_u^{-1})$ are nonnegative. Thus we may set $L=1$, completing the proof.
\end{proof}

For ease of reference, we record the following direct corollary.
\begin{corollary}
For any point $u$, we have
\[
\|y_u - A^{-1}_uq\|  \leq \|A^{-1}_u\| \|A_u y_u - q\| \leq \frac{\|A^{-1}_u\|\|\nabla (f\circ A^{-1}_u)(q)\|}{\lambda}.
\]
\end{corollary}

Next we will compute the Hessian of $\varphi(u,y)$, and use it to show that the norm of the Hessian of $\widetilde \varphi$ is bounded by a constant independent of $\lambda$. Defining 
\[
R(u,y,v) = \nabla_u \left[ G(u,y)^Tv \right] ~\textrm{ and }~ K(u,v) = \nabla_u\left[A_u^Tv\right],
\]
we can partition the Hessian as follows:
\[
\nabla^2 \varphi = 
\begin{bmatrix}
\varphi_{uu} & \varphi_{uy} \\
\varphi_{yu} & \varphi_{yy}
\end{bmatrix}
\]
where 
\[
\begin{aligned}
\varphi_{uu}(u, y) & = \lambda \left(G(u,y)^T G(u,y) + R(u, y, A_u y-q) \right), \\
\varphi_{yy}(u,y) & = \nabla^2 f (y) + \lambda A_u^TA_u, \\
\varphi_{yu}(u,y) & = \lambda \left(K(u, A_u y - q) + A_u^TG(u,y)\right).
\end{aligned}
\]
See \cite[Section 4]{van2015penalty} for more details.
Moreover, it is known that the Hessian of the reduced function $\widetilde \varphi$ admits the expression \cite[Equation 22]{van2015penalty}
\begin{equation}\label{eqn:schur}
\nabla^2 \widetilde\varphi (u)  = \varphi_{uu}(u,y_u) - \varphi_{uy}(u, y_u)\varphi_{yy}(u,y_u)^{-1} \varphi_{yu}(u,y_u),
\end{equation}
which is simply the Schur complement of $\varphi_{yy}(u,y_u)$ in $\nabla^2 \varphi(u,y_u)$.
We define the operator norms
\[
\begin{aligned}
\left\| \nabla_u G(u,y)^T \right\| &:= \sup_{\|v\|\leq 1} \left \| \nabla_u \left[ G(u,y)^Tv \right] \right\|, \\
\left\| \nabla_u A_u^T \right\| &:= \sup_{\|v\|\leq 1} \left \| \nabla_u \left[ A_u^Tv \right] \right\|.
\end{aligned}
\]
Using this notation, we can prove the following key bounds.

\begin{corollary}\label{lem:key_ineq}
%The following quantities are bounded above by $\frac{\|\nabla (f\circ A^{-1}_u)(q)\|}{\lambda}$
%\[
%\|\varphi_{yy} (u,y_u)^{-1}\|, \quad \|K(u, A_uy_u - q)\|, \quad \|R(u, y_u, A_uy_u - q)\|. 
%\]
%\begin{proof}
For any points $u$ and $y$, the inequalities hold:
\begin{align*}
\|\varphi_{yy}(u,y)^{-1}\| &\leq \frac{ \|A_u^{-1}\|^2}{\lambda},\\
\|R(u, y_u, A_u y_u - q)\| &\leq \frac{\|\nabla (f\circ A^{-1}_u)(q)\|\left\|\nabla_uG(u,y_u)\right\|}{\lambda},\\
\|K(u, A_u y_u - q)\| &\leq \frac{\|\nabla (f\circ A^{-1}_u)(q)\|\left\| \nabla_u  A_u^T\right\|}{\lambda} .
\end{align*}
%\end{proof}
\end{corollary}
\begin{proof}
The first bound follows by the inequality 
\[
\begin{aligned}
\|\varphi_{yy}(u,y)^{-1}\| &= \frac{1}{\lambda}\left\|A_u^{-1}\left(\frac{1}{\lambda} A_u^{-T}\nabla^2 f(y)A_u^{-1} + I\right)^{-1}A_u^{-T}\right\| \\
& \leq \frac{ \|A_u^{-1}\|^2}{\lambda},
\end{aligned}
\]
and the remaining bounds are immediate from Lemma~\ref{lem:resBound}.
\end{proof}

Next, we need the following elementary linear algebraic fact. 
%We provide a short proof for completeness.
\begin{lemma}
\label{lem:LA}
For any positive semidefinite matrix $B$ and a real $\lambda >0$, we have 
$\|I - \left(I + \frac{1}{\lambda}B\right)^{-1}\|_2\leq \frac{\|B\|}{\lambda}$.
\end{lemma}
\begin{proof}
Define the matrix
$
F = I - \left(I + \frac{1}{\lambda}B\right)^{-1}
$
and consider an arbitrary point $z$. Observing the inequality $\|\left(I + \frac{1}{\lambda}B\right)^{-1}\|\leq 1$ and
defining the point
 $p := (I + \frac{1}{\lambda} B)z$, we obtain 
\[
\begin{aligned}
\|Fz\| & = \left\|z - \left(I + \frac{1}{\lambda}B\right)^{-1}z \right\|  \\
& =\left \| \left(I + \frac{1}{\lambda}B\right)^{-1}p  - \left(I + \frac{1}{\lambda}B\right)^{-1}z\right\| \\
&  \leq \left\| \left(I + \frac{1}{\lambda}B\right)^{-1} \right\| \|p-z\| \\ 
& \leq \left\|\frac{1}{\lambda} Bz\right\| \leq  \frac{\|B\|}{\lambda}\|z\|.
\end{aligned}
\]
Since this holds for all $z$, the result follows.
\end{proof}

Putting all the pieces together, we can now prove the main theorem of this section.

\begin{theorem}[Norm of the reduced Hessian]\label{thm:main_result}
The operator norm of $\nabla^2 \widetilde\varphi(u)$ is bounded by a quantity $C(A(\cdot),f,q,u)$ independent of $\lambda$. 
\end{theorem}

\begin{proof}	
To simplify the proof, define $G:=G(u,y_u)$, $R:= R(u,y_u,A_uy-q)$, $K:= K(u, A_uy_u - q)$, 
and $\Delta = \varphi_{yy}(u,y_u)$. 
When $\lambda \leq 1$, the operator norm of $\nabla^2 \widetilde \phi$ 
has a trivial bound directly from equation~\eqref{eqn:schur}.
For large $\lambda$, 
after rearranging \eqref{eqn:schur}, we can write $\nabla^2 \widetilde\varphi$ as follows:
\begin{equation}
\label{eq:hessian_a}
\begin{aligned}
%\nabla^2 \widetilde\varphi (u) 
& \lambda R - \lambda^2 \left(K^T\Delta^{-1}K
+  K^T\Delta^{-1} A_u^TG + G^TA_u \Delta^{-1} K \right) \\
& + \lambda G^T G - \lambda^2 G^T A_u\Delta^{-1}A_u^TG.
\end{aligned}
\end{equation}
Corollary~\ref{lem:key_ineq} implies that  
the operator norm of the first row of~\eqref{eq:hessian_a} is bounded above by the quantity
\begin{equation}\label{eqn:basic_bound}
\begin{aligned}
L_u\|\nabla_u G\|&+\frac{1}{\lambda}L_u^2\|\nabla_u A_u^T\|^2\|A_u^{-1}\|^2 \\
&+2 L_u\|\nabla_u A_u^T\|\|A_u^{-1}\|^2\|A_u^TG\|,
\end{aligned}
\end{equation}
where we set $L_u:=\|\nabla (f\circ A_u^{-1})(q)\|$.
Notice that the expression in \eqref{eqn:basic_bound} is independent of $\lambda$. We rewrite the second row of~\eqref{eq:hessian_a} using
the explicit expression for $\Delta$: 
\[
\begin{aligned}
 \lambda G^T G& - \lambda^2 G^T A_u\Delta^{-1}A_u^TG \\
% &= \lambda G^T G - \lambda G^TA_u\left(A_u^{-1}\left(\frac{1}{\lambda} A_u^{-T}\nabla^2 f(y_u)A_u^{-1} + I\right)^{-1}A_u^{-T}\right)A_u^TG \\
&= \lambda \left( G^T G - G^T\left(\frac{1}{\lambda} A_u^{-T}\nabla^2 f(y_u)A_u^{-1} + I\right)^{-1}G\right) \\
& = 
\lambda \left( G^T\left( I - \left(\frac{1}{\lambda} A_u^{-T}\nabla^2 f(y_u)A_u^{-1} + I\right)^{-1}\right)G\right).
\end{aligned}
\]
Applying Lemma~\ref{lem:LA} with $B = A_u^{-T}\nabla^2 f(y_u)A_u^{-1}$, we have 
\[
\left\| \lambda G^T G - \lambda^2 G^T A_u\Delta^{-1}A_u^TG \right\| \leq \|G\|^2 \|A_u^{-T}\nabla^2 f(y_u)A_u^{-1}\|.
\]
Setting
\[
\begin{aligned}
C(A(\cdot),f,q,u) &:= L_u\|\nabla_u G\|+ 2 L_u\|\nabla_u A_u^T\|\|A_u^{-1}\|^2\|A_u^TG\| \\
& + \|G\|^2 \|A_u^{-1}\|^2\|\nabla^2 f(y_u)\| \\ 
& + \frac{1}{\lambda}L_u^2\|\nabla_u A_u^T\|^2\|A_u^{-1}\|^2.
\end{aligned}
\]
For $\lambda > 1$, the last term is always trivially bounded by $L_u^2\|\nabla_u A_u^T\|^2\|A_u^{-1}\|^2$ and the result follows. 

\end{proof}

\subsection{Inexact analysis of the projection subproblem}
 In practice, one can rarely evaluate $\widetilde{\varphi}(u)$ exactly. It is therefore important to understand how inexact solutions of the inner problems \eqref{eqn:subprob_eff} impact iteration complexity of the outer problem. The results presented in the previous section form the foundation for such an analysis. 
For simplicity,  we assume that $g$ is smooth, though the results can be generalized, as we comment on shortly. 

In this section, we compute the overall complexity of the partial minimization technique when the outer nonconvex minimization problem \eqref{eqn:reduced} is solved by an inexact gradient descent algorithm. When $g$ is nonsmooth, a completely analogous analysis applies to the prox-gradient method.
We only focus here on gradient descent, as opposed to more sophisticated methods, since the analysis is straightforward. We expect quasi-Newton methods and limited memory variants to exhibit exactly the same behavior (e.g. Figure~\ref{fig:boundaryComp}). We do not perform a similar analysis here for inexact quasi-Newton methods, as the global efficiency estimates even for exact quasi-Newton methods for nonconvex problems are poorly understood.

Define the function $H(u):=g(u)+\widetilde\varphi(u)$. Let $\beta>0$ be the Lipschitz constant of the gradient $\nabla H=\nabla g+\nabla \widetilde\varphi$. Fix a constant $c >0$, and suppose that in each iteration $k$, we compute a vector $v_k$ with $\|v_k-\nabla H(u_k)\|\leq \frac{c}{k}$. Consider then the inexact gradient descent method $u_{k+1}=u_k-\frac{1}{\beta}v_k$. Then we deduce
\begin{equation}
\label{eq:inequality}
\begin{aligned}
H(u_{k+1})-H(u_k)&\leq - \langle\nabla H(u_k),\beta^{-1}v_k\rangle+\frac{\beta}{2}\|\beta^{-1}v_k\|^2\\
&=\frac{1}{2\beta}\left( \|v_k-\nabla H(u_k)\|^2-\|\nabla H(u_k)\|^2\right).
\end{aligned}
\end{equation}
Hence we obtain the convergence guarantee:
\begin{align*}
\min_{i=1,\ldots,k} &\|\nabla H(u_i)\|^2\leq\frac{1}{k}\sum_{i=1}^k \|\nabla H(u_i)\|^2 \\
&\leq \frac{2\beta \left(H(u_1)-H^*\right)}{k}  +\frac{1}{k}\sum_{i=1}^k \|v_k-\nabla H(u_k)\|^2 \\
&\leq \frac{2\beta \left(H(u_1)-H^*\right)}{k} + \frac{c^{2}\pi^2}{6k}\leq \frac{\beta^2 \left\|u_1-u^*\right\|^2+c^2\pi^2/6}{k}.
\end{align*}
where~\eqref{eq:inequality} is used to go from line 1 to line 2. 
Now, if we compute $\nabla g$ exactly, the question is how many inner iterations are needed to guarantee 
$\|v_k-\nabla H(u_k)\|\leq \frac{c}{k}$. For fixed $u=u_k$, the inner objective is  
\[
\varphi(u_k,y) = f(y) + \frac{\lambda}{2} \|A(u_k)y-q\|^2.
\]
The condition number (ratio of Lipschitz constant of the gradient 
over the strong convexity constant) of $\varphi(u_k,y)$ in $y$
is 
\[
\kappa_k := \frac{1}{\lambda}\mathrm{Lip}(f) \|A(u_k)^{-1}\|^2 + \|A(u_k)\|^2\|A(u_k)^{-1}\|^2.
\] 
Notice that $\kappa_k$ converges to the squared condition number of $A(u_k)$ as $\lambda \uparrow \infty$. Gradient descent on the function $\varphi(u_k,\cdot)$
guarantees
\(
\|y_i - y^*\|^2 \leq \epsilon 
\)
after $\kappa_k \log\left(\frac{\|y_0-y^*\|^2}{\epsilon}\right)$ iterations.
Then we have 
\[
\|\nabla_u \varphi(u_k, y_i) - \nabla_u \varphi(u_k,y^*)\| 
\leq
\lambda\|\nabla_u G(u, y^*)\| \|y_i - y^*\|.
\]
Since we want the left hand side to be bounded by $\frac{c}{k}$, we simply need to ensure
\[
\|y_i - y^*\|^2 \leq \frac{c^2}{k^2 \lambda ^2\|\nabla_u G(u_k,y^*)\|^2}.
\]
Therefore the total number of inner iterations is no larger than 
\[
\kappa_k \log\left(
\frac{\|y_0 - y^*\|^2 \|\nabla_u G(u_k,y^*)\|^2}{c^2}k^2\lambda^2
\right),
\]
which grows very slowly with $k$ and with $\lambda$. In particular, the number of iterations to solve the inner problem scales as $\log(k\lambda )$ to achieve a global 
$\frac{1}{k}$ rate in $\|\nabla H\|^2$. If instead we use a fast-gradient method \cite[Section 2.2]{nest_lect_intro} for minimizing $\varphi(u_k,\cdot)$, we can replace $\kappa_k$ with the much better quantity $\sqrt{\kappa}_k$ throughout.

\section{Numerical Illustrations}\label{sec:num_ill}

In this section, we present two  representative examples of PDE constrained optimization 
(boundary control and optimal transport) and a problem of robust dynamic inference. In each case, we show that practical experience
supports theoretical results from the previous section. In particular, in each numerical 
experiment, we study the convergence behavior of the proposed method 
as $\lambda$ increases. 

\subsection{Boundary control}
%An important example is boundary control, where $u$ represents control parameters that appear in a partial differential %equation, and $g$ measures the deviation between the solution of the PDE and some desired state or output, for example 
%\[
%\min_{y,u} \rho(Py - d) \quad \mbox{s.t.} \quad  A(u)y = q;
%\]
%where $A(u)$ is the differential operator, $q$ is the input, and $y$ is the state. 

In boundary control, the goal is to steer a system towards a desired state by controlling its boundary conditions. Perhaps the simplest example of such a problem is the following. Given a source $q(x)$, defined on a domain $\Omega$, we seek boundary conditions $u$ such that the solution to the Poisson problem
\begin{equation*}
\begin{aligned}
\Delta y &= q\quad \textrm{ for }x\in \Omega \\
y|_{\partial\Omega} &= u
\end{aligned}
\end{equation*}
is close to a desired state $y_d$. Discretizing the PDE yields the system
\[
Ay + Bu = q
%\left(
%\begin{matrix}
%	A & B \\
%	0 & I 
%\end{matrix}
%\right)
%\left(
%\begin{matrix}
%	y_i \\
%	y_b
%\end{matrix}
%\right)
%= 
%\left(
%\begin{matrix}
%	q \\
%	u
%\end{matrix}
%\right),
\]
where $A$ is a discretization of the Laplace operator on the interior of the domain and $B$ couples the interior gridpoints
to the boundary. The corresponding PDE-constrained optimization
problem is given by 
\[
\min_{u,y}~ \textstyle{\frac{1}{2}}\|y - y_d\|_2^2 \qquad \textrm{subject to} \qquad Ay + Bu = q,
\]
whereas the penalty formulation reads
\[
\min_{u,y}~ \textstyle{\frac{1}{2}}\|y - y_d\|_2^2 + \textstyle{\frac{\lambda}{2}} \|Ay + Bu - q \|_2^2.
\]
Since both terms are quadratic in $y$, we can quickly solve for $y$ explicitly.

\subsubsection{Numerical experiments}
In this example, we consider an L-shaped domain with a source $q$ shaped like a Gaussian bell, as depicted in figure \ref{fig:boundary11}. 
Our goal is to get a constant distribution $y_d = 1$ in the entire domain. The solution for $u=1$ is shown in figure \ref{fig:boundary12} (a). To solve the optimization problem we use a steepest-descent method with a fixed step-size, determined from the Lipschitz constant of the gradient. The result of the constrained formulation is shown in figure \ref{fig:boundary12} (b). We see that by adapting the boundary conditions we get a more even distribution. The convergence behavior for various values of $\lambda$ is shown in figure \ref{fig:boundary13} (a). We see that as $\lambda\uparrow\infty$, the behaviour tends towards that of the constrained formulation, as expected. The Lipschitz constant of the gradient (evaluated at the initial point $u$), as a function of $\lambda$ is shown in figure \ref{fig:boundary13} (b); the curve levels off as the theory predicts. 

\begin{figure}[h!]
	\centering
	\begin{tabular}{c}
			\includegraphics[scale=.3]{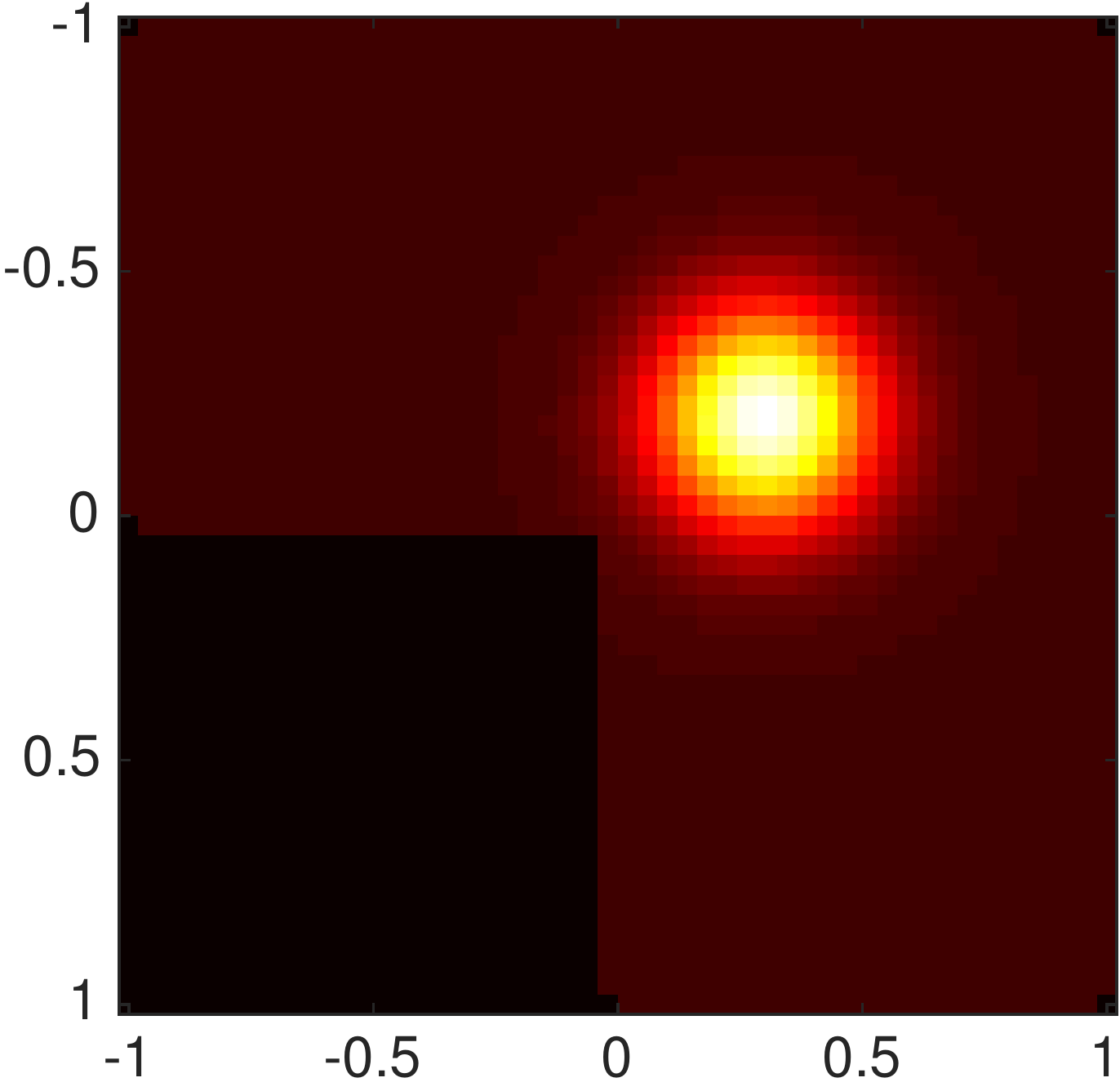}\\
	\end{tabular}
	\caption{L-shaped domain with the source function $q$.}
	\label{fig:boundary11}
\end{figure}

\begin{figure}[h!]
	\centering
	\begin{tabular}{c}
			\includegraphics[scale=.3]{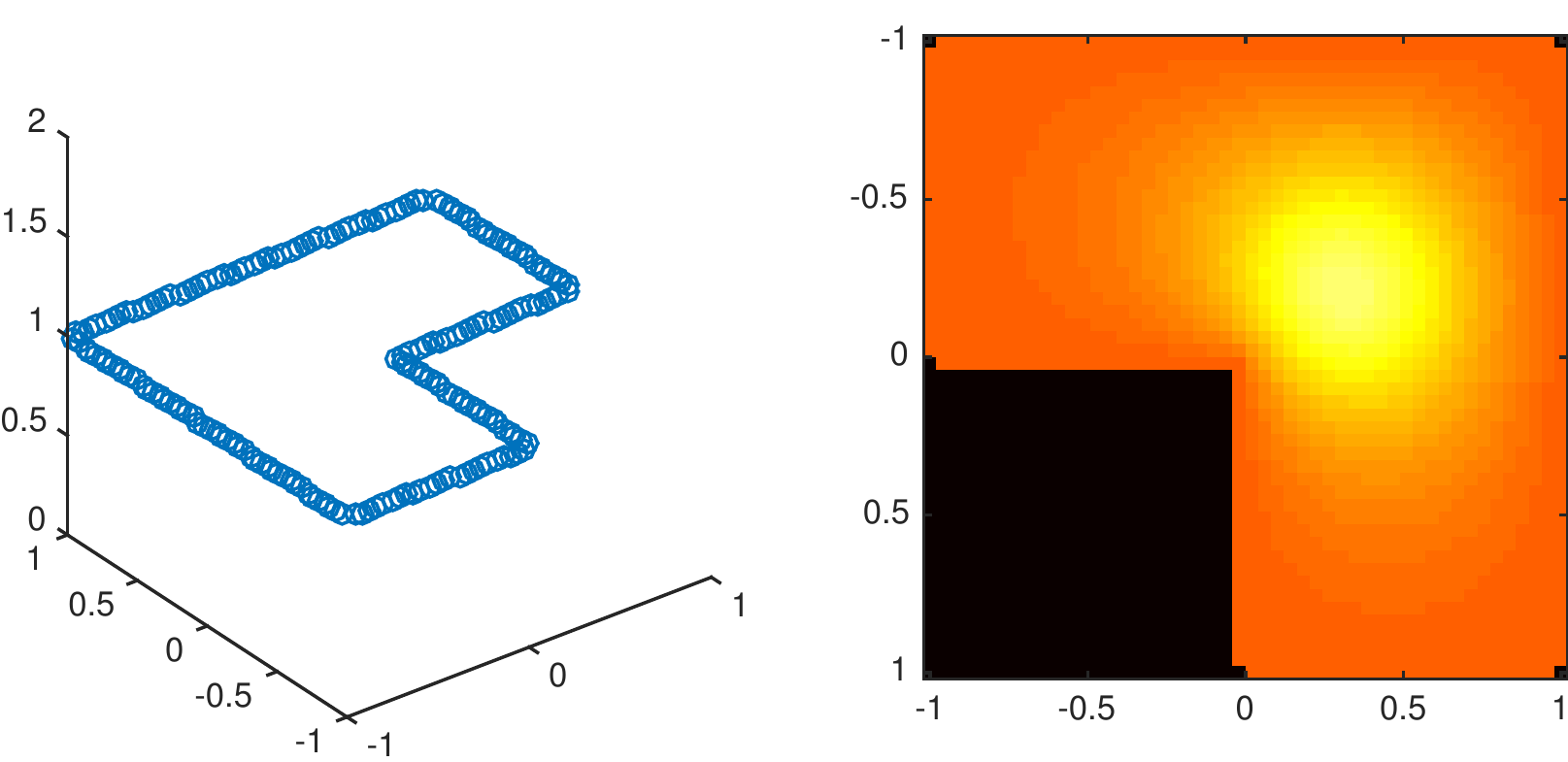}\\
			{\small (a)}\\
			\includegraphics[scale=.3]{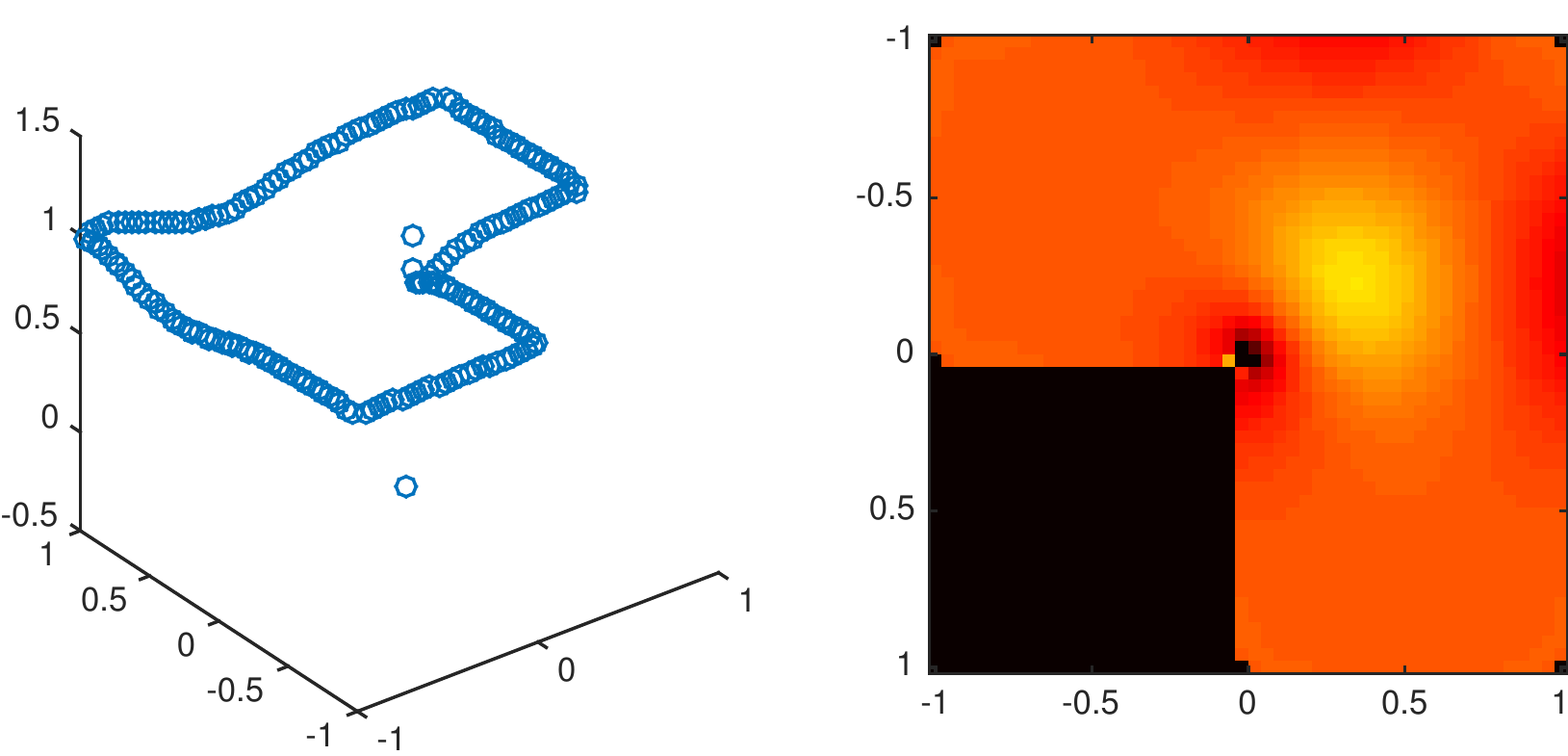}\\
			{\small (b)}\\
	\end{tabular}
	\caption{Boundary values, $u$, and solution in the interior for the initial and optimized boundary values are depicted in (a) and (b) respectively.}
	\label{fig:boundary12}
\end{figure}

\begin{figure}[h!]
\hspace{-.3in}
	\begin{tabular}{ccc}
			\includegraphics[scale=.27]{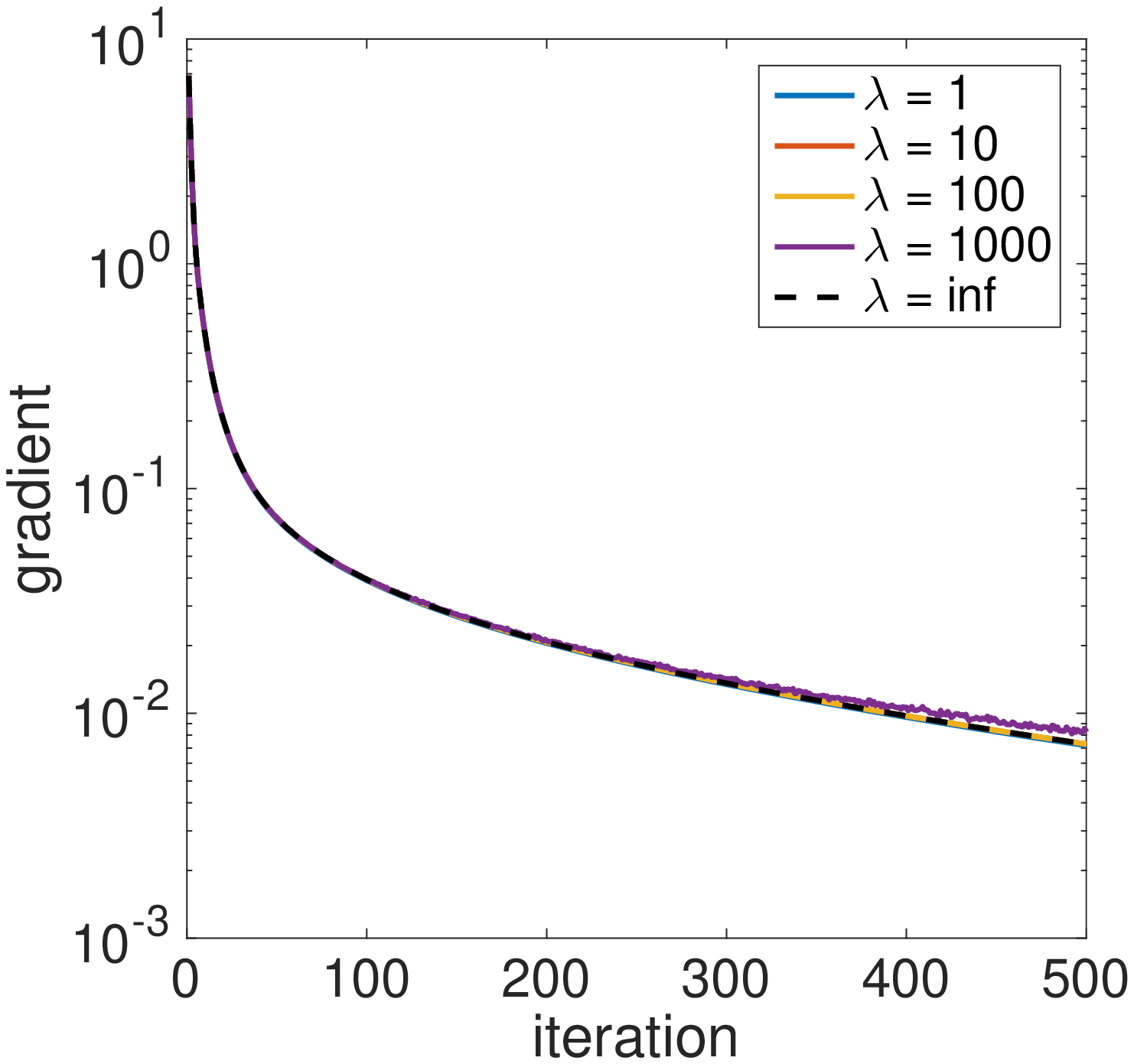}& \hspace{-.6in}
			\includegraphics[scale=.27]{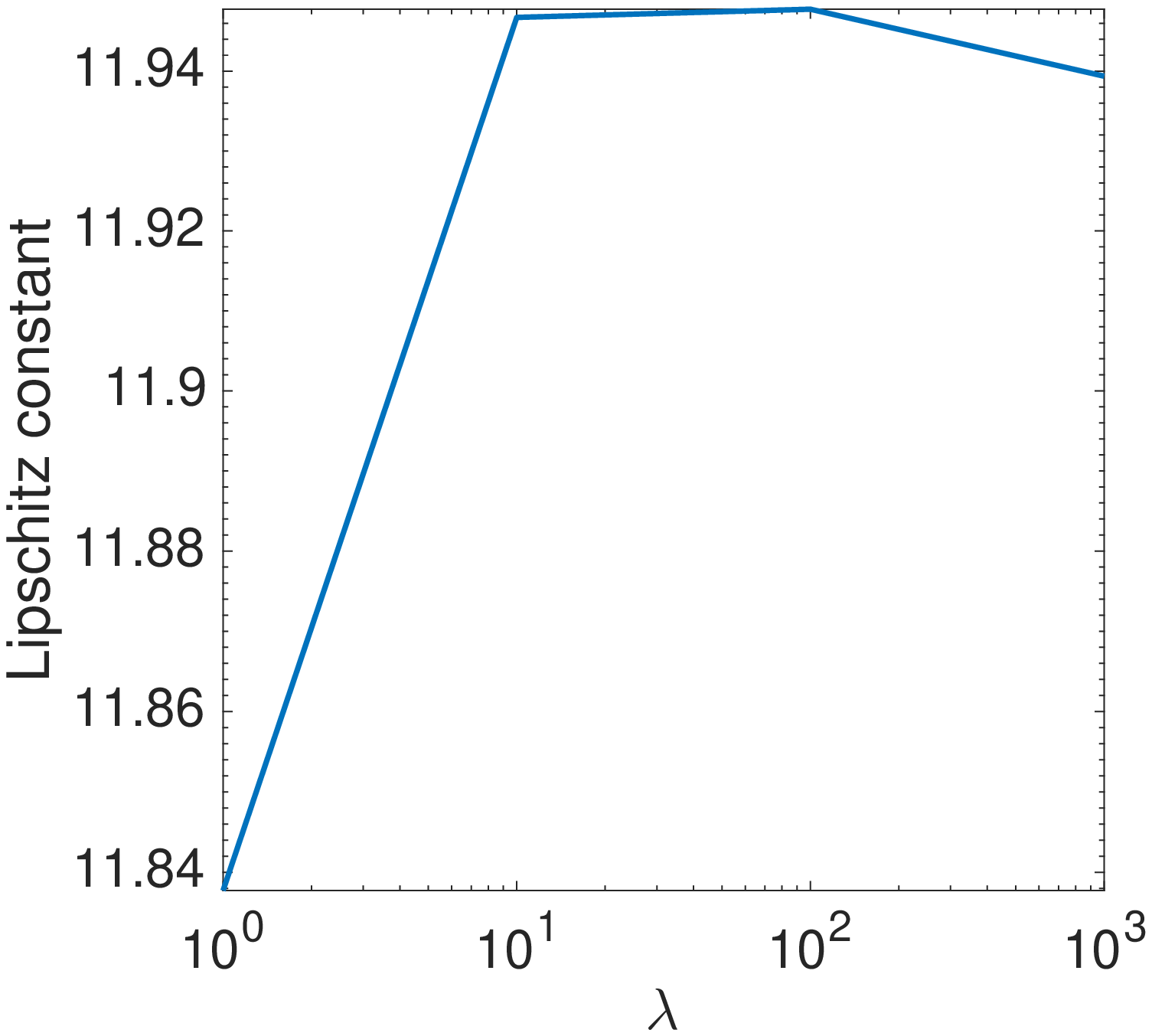}\\
			{\small (a)}&{\small (b)}\\
	\end{tabular}
	\caption{The convergence plots for various values of $\lambda$ are depicted in (a), while (b) shows the dependence of the (numerically computed) Lipschitz constant on $\lambda$. }
	\label{fig:boundary13}
\end{figure}
\subsection{Optimal transport}
\begin{figure}
	\centering
	\includegraphics[scale=.5]{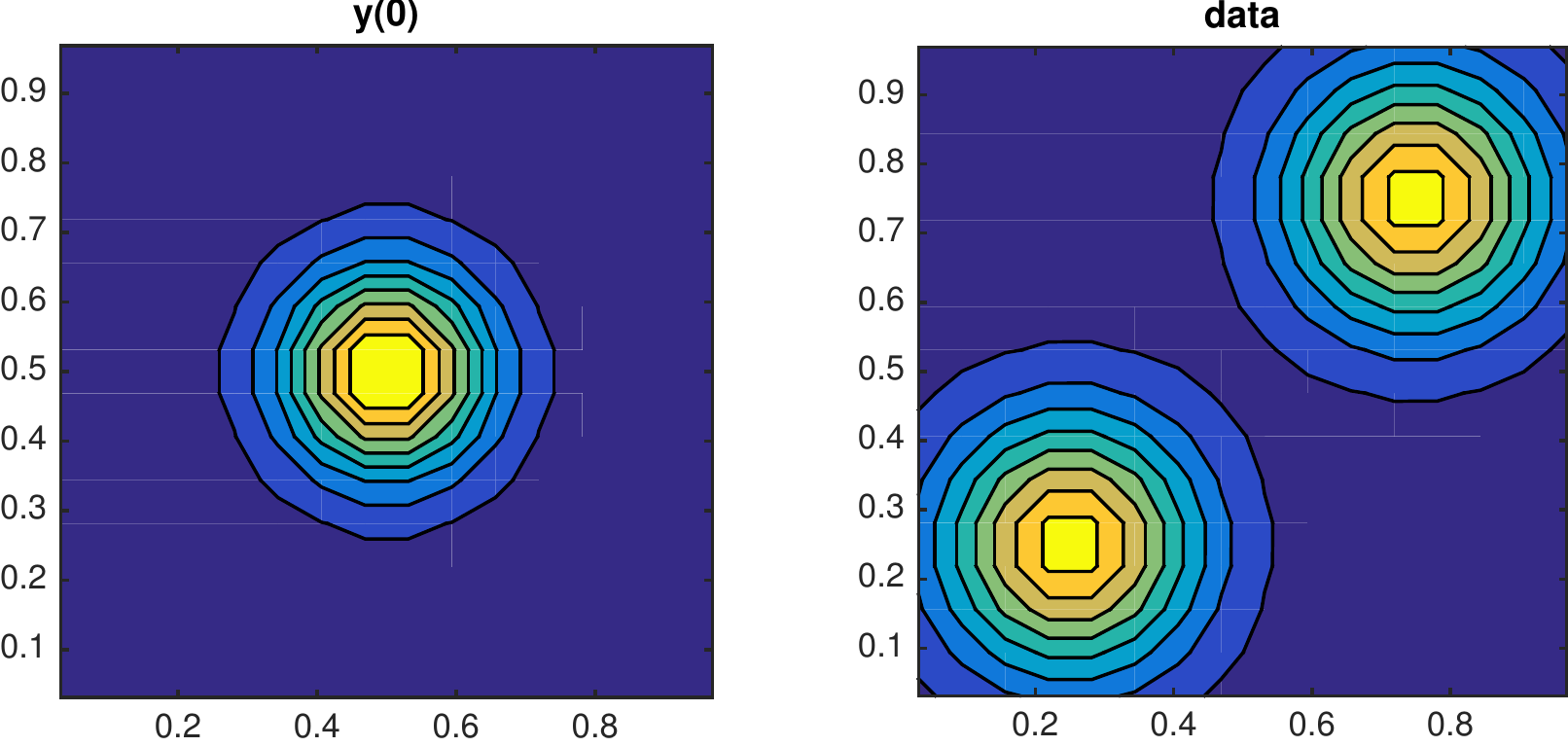}\\
	\caption{The initial (left) and desired (right) mass density are shown.}
	\label{fig:transport_11}
\end{figure}

\begin{figure}
	\centering
	\includegraphics[scale=.5]{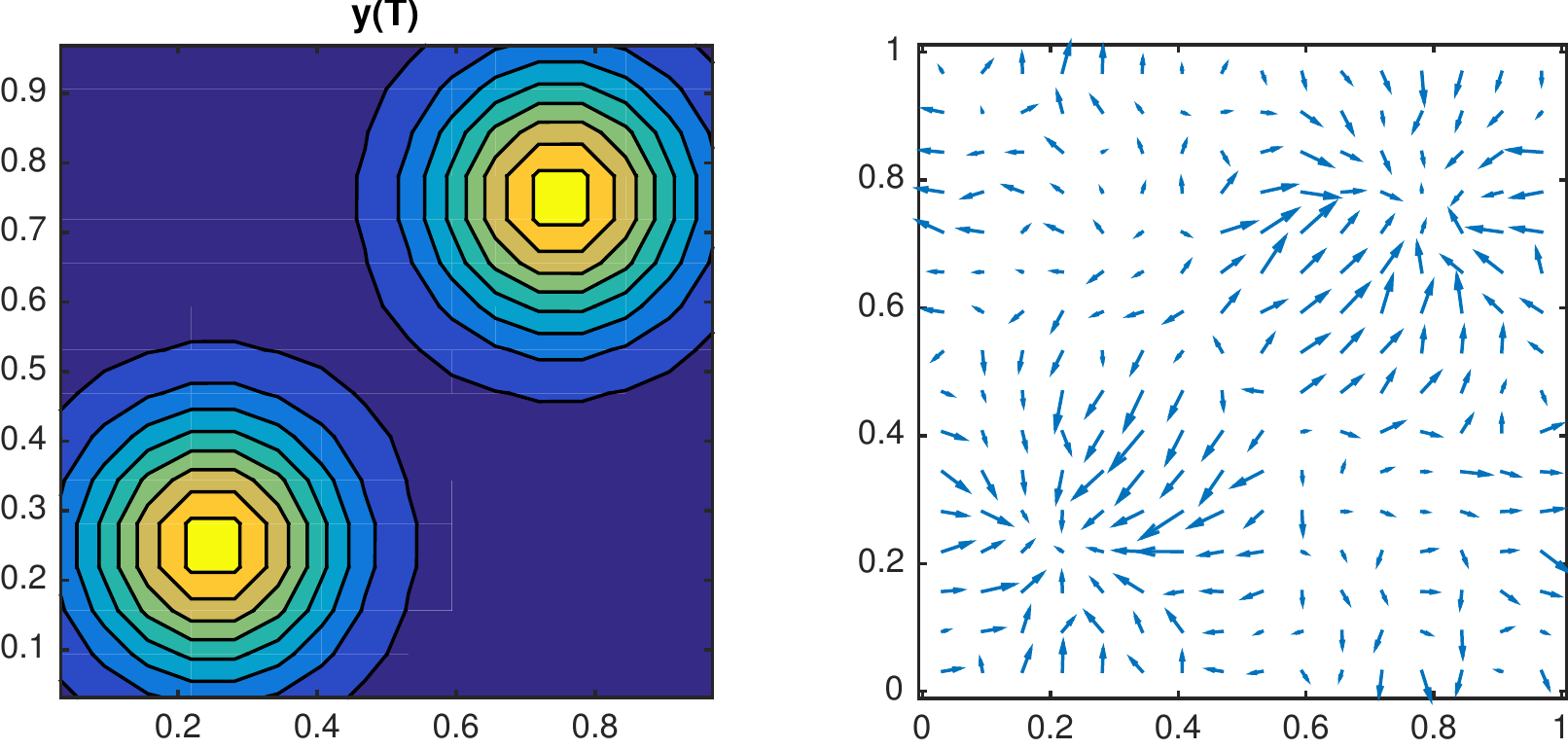}
	\caption{The mass density obtained after optimization (left) and the corresponding time-averaged flow field (right) are shown.}
	\label{fig:transport_12}
\end{figure}

\begin{figure}
\hspace{-.4in}
	\begin{tabular}{cc}
		\includegraphics[scale=.29]{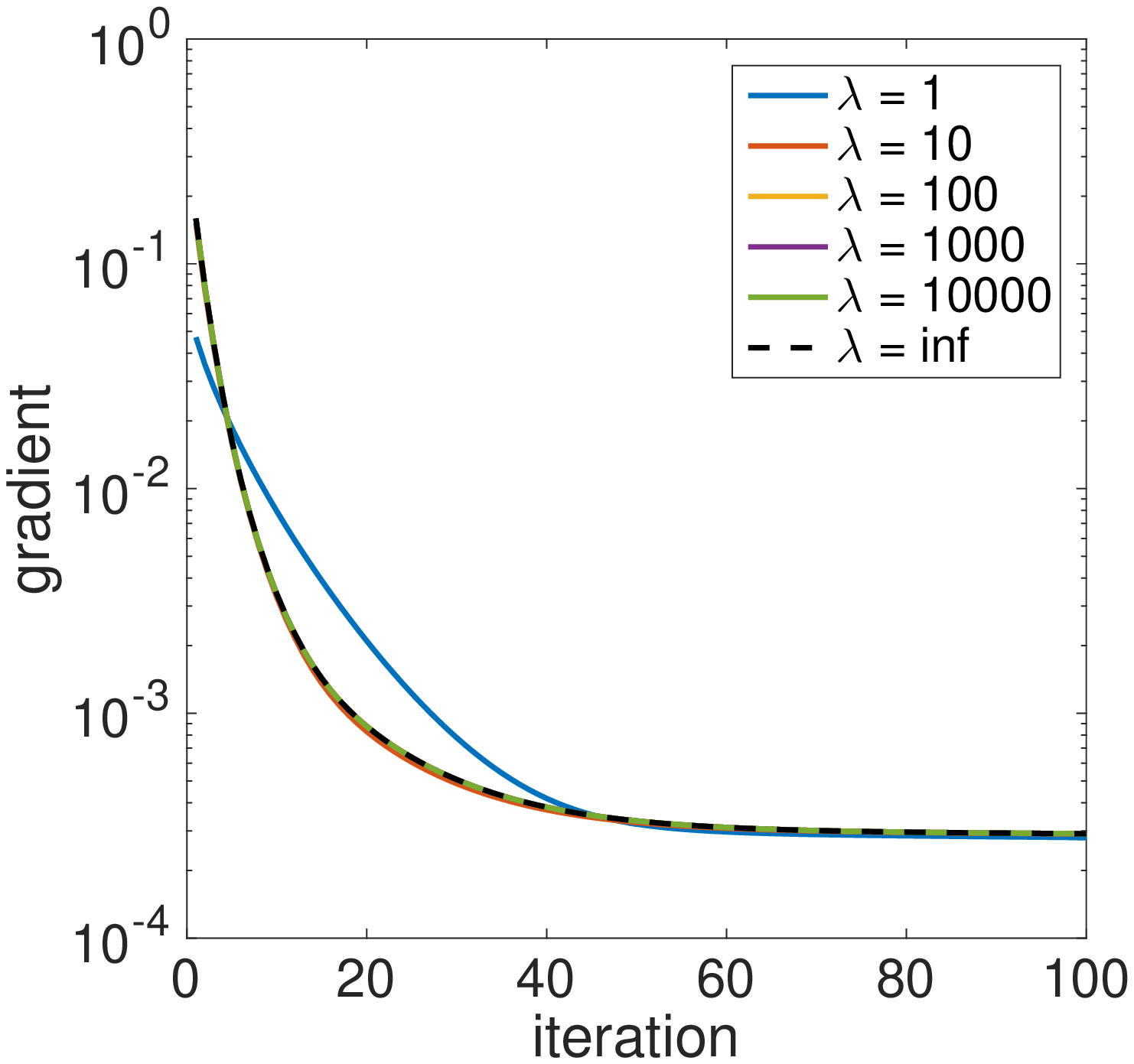}&\hspace{-.6in}
		\includegraphics[scale=.29]{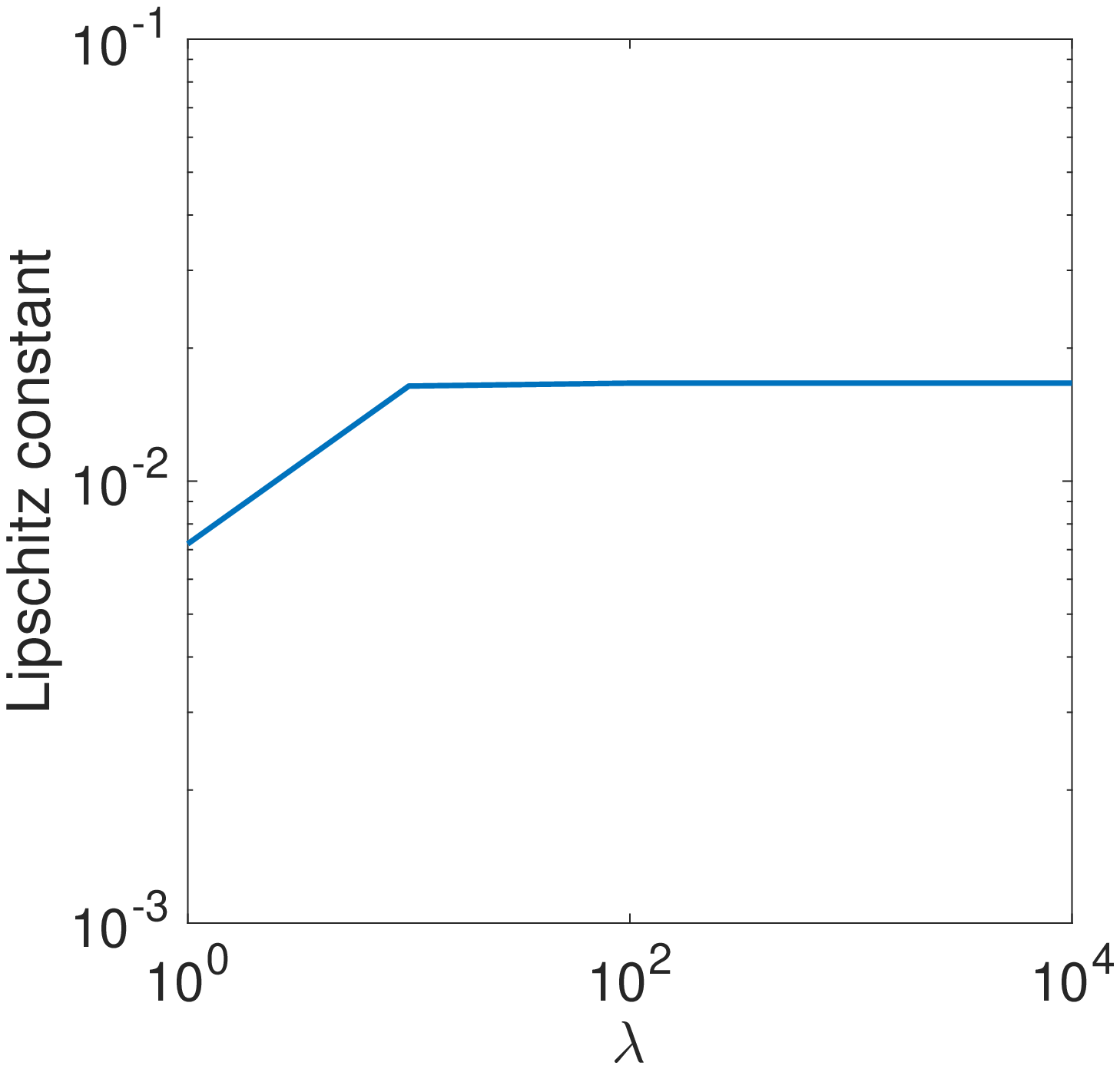}\\
		{\small (a)}&{\small (b)}\\
	\end{tabular}
	\caption{Convergence plots for various values of $\lambda$ are shown in (a), while (b) shows the (numerically computed) Lipschitz constant as a function of $\lambda$.}
	\label{fig:transport_13}
\end{figure}

The second class of PDE-constrained problems we consider comes from optimal transport, where the goal is to determine a mapping, or flow, that optimally transforms one mass density function into another. Say we have two density functions $y_0(x)$ and $y_T(x)$, with $x\in\Omega$, we can formulate the problem as finding a flowfield, $u(t,x) = \left(\begin{matrix} u_1(t,x)\\u_2(t,x) \end{matrix}\right)$, such that $y_T(x) = y(T,x)$ and $y_0(x)  = y(0,x)$, where $y(t,x)$ solves
\[
y_t + \nabla\cdot(y u) = 0.
\]
Discretizing using an implicit Lax-Friedrichs scheme \cite{Haber2007}, the PDE reads
\[
A(u)y = q,
\]
where $q$ contains the initial condition and we have
\[
\small
A(u) = 
\left(
\begin{matrix}
	I + \Delta t B(u^1) &                     &    &                     \\
	-M                  & I + \Delta t B(u^2) &    &                     \\
	                    & -M                  &    &                     \\
						&                     &    &                     \\
						&                     & -M & I + \Delta t B(u^N) \\
\end{matrix}
\right),
\]
with $M$ a four-point averaging matrix and $B$ containing the discretization of the derivative terms.
Adding regularization to promote smoothness of $u$ and $y$ in time~\cite{Haber2007}, we obtain
the problem
\begin{equation}
\label{eq:flowCon}
\begin{aligned}
\min_{u,y}~ \textstyle{\frac{1}{2}} \|Py - y_T\|_2^2 + \textstyle{\frac{\alpha^2}{2}} y^TL\mathsf{diag}(u)u \\
\textrm{subject to}\qquad A(u)y = q.
\end{aligned}
\end{equation}
Here, $P$ restricts the solution $y$ to $t=T$, $\alpha$ is a regularization parameter and $L$ is a block matrix with $I$ 
on the main and upper diagonal.
The penalized formulation is
\begin{equation}
\label{eq:flowConPen}
\min_{u,y}~ \textstyle{\frac{1}{2}} \|Py - y_T\|_2^2 + \textstyle{\frac{\alpha^2}{2}} y^TL\mathsf{diag}(u)u + \frac{1}{2}\lambda \|A(u)y - q\|^2.
\end{equation}
Again the partial minimization in $y$ amount to minimizing a quadratic function. 

\subsubsection{Numerical experiments}
For the numerical example we consider the domain $\Omega = [0,1]^2$, discretized with $\Delta x=1/16$ and $T = 1/32$ 
with a stepsize of $\Delta t = 1/8$. The initial and desired state are depicted in figure \ref{fig:transport_11}. The resulting
 state obtained at time $T$ and the corresponding time-averaged flowfield are depicted in figure \ref{fig:transport_12}. The initial flow $u_0$ was generated by i.i.d. samples from a standard Gaussian random variable. To minimize~\eqref{eq:flowConPen}, we used a steepest-descent method with constant step size, using the largest eigenvalue of the Gauss-Newton Hessian at the initial $u$ as an estimate of the Lipschitz constant. The convergence behavior for various values of $\lambda$ as well as the corresponding estimates of the Lipschitz constant at the final solution are shown in figure \ref{fig:transport_13}.

\begin{figure} 
\centering
%%%%%%%%%%%%%%%%%%%%%
% Densities
%%%%%%%%%%%%%%%%%%%%%
\begin{tikzpicture}
  \begin{axis}[
    thick,
    width=.2\textwidth, height=2cm,
    xmin=-4,xmax=4,ymin=0,ymax=1,
    no markers,
    samples=50,
    axis lines*=left, 
    axis lines*=middle, 
    scale only axis,
    xtick={-1,1},
    xticklabels={},
    ytick={0},
    ] 
%\addplot[red,domain=-2:-1,densely dashed]{-x-.5};
\addplot[domain=-4:+4,densely dashed]{exp(-.5*x^2)/sqrt(2*pi)};

 \addplot[red, domain=-4:-1]{0.35*exp(-(abs(x)-.5))};
 \addplot[red, domain=-1:1]{0.35*exp(-.5*x^2)};
  \addplot[red, domain=1:+4]{0.35*exp(-(abs(x)-.5))};

  \addplot[blue, domain=-4:+4, dotted]{0.3*exp(-.5*ln(1 + x^2))};
%\addplot[red,domain=+1:+2,densely dashed]{x-.5};
%\addplot[blue,mark=*,only marks] coordinates {(-1,.5) (1,.5)};
  \end{axis}
\end{tikzpicture}
%%%%%%%%%%%%%%%%%%%%%
% Penalties
%%%%%%%%%%%%%%%%%%%%%
\begin{tikzpicture}
  \begin{axis}[
    thick,
    width=.2\textwidth, height=2cm,
    xmin=-3,xmax=3,ymin=0,ymax=2,
    no markers,
    samples=50,
    axis lines*=left, 
    axis lines*=middle, 
    scale only axis,
    xtick={-1,1},
    xticklabels={},
    ytick={0},
    ] 
%\addplot[red,domain=-2:-1,densely dashed]{-x-.5};
\addplot[domain=-3:+3,densely dashed]{.5*x^2};
 \addplot[red, domain=-3:-1]{abs(x)-.5};
 \addplot[red, domain=-1:1]{.5*x^2};
  \addplot[red, domain=1:+3]{abs(x)-.5};
  \addplot[blue, domain=-3:+3, dotted]{.5*ln(1 + x^2)};
%\addplot[red,domain=+1:+2,densely dashed]{x-.5};
%\addplot[blue,mark=*,only marks] coordinates {(-1,.5) (1,.5)};
  \end{axis}
\end{tikzpicture}
%%%%%%%%%%%%%%%%%%%%%
% Influence Functions
%%%%%%%%%%%%%%%%%%%%%
%\begin{tikzpicture}
%  \begin{axis}[
%    thick,
%    width=.3\textwidth, height=2cm,
%    xmin=-3,xmax=3,ymin=-2,ymax=2,
%    no markers,
%    samples=50,
%    axis lines*=left, 
%    axis lines*=middle, 
%    scale only axis,
%    xtick={-1,1},
%    xticklabels={},
%    ytick={0},
%    ] 
%%\addplot[red,domain=-2:-1,densely dashed]{-x-.5};
%\addplot[domain=-3:3,densely dashed]{.5*x};
% \addplot[red, domain=-3:-1]{-.5};
% \addplot[red, domain=-1:1]{.5*x};
%  \addplot[red, domain=1:3]{.5};
%\addplot[blue, domain=-3:3, dotted]{x/(1 + x^2)};
%%\addplot[red,domain=+1:+2,densely dashed]{x-.5};
%%\addplot[blue,mark=*,only marks] coordinates {(-1,.5) (1,.5)};
%  \end{axis}
%\end{tikzpicture}
    \caption{\label{GLT-KF}Left: Densities, Gaussian (black dash), Huber (red solid), and Student's t (blue dot). Right:  Negative Log Likelihoods.}
%\caption{
%Densities $\mathbf{p}(\textcolor{blue}{v})$, penalties $-\ln\mathbf{p}(\textcolor{blue}{v})$ ,
 %and influence fns $-\frac{d}{d \textcolor{blue}{v} }\ln\mathbf{p}(\textcolor{blue}{v})$. }
\end{figure}

\subsection{Robust dynamic inference with the penalty method}\label{sec:robust}

In many settings, data is naturally very noisy, and 
a lot of effort must be spent in pre-processing and cleaning before applying
standard inversion techniques. 

To narrow the scope, consider dynamic inference,
where we wish to infer both hidden states  and unknown 
parameters driven by an underlying ODE.  
Recent efforts have focused on developing inference formulations 
that are robust to outliers in the data~\cite{Aravkin2011,Farahmand2011,aravkin2014robust}, 
using convex penalties such as
$\ell_1$, Huber~\cite{Hub} and non-convex penalties such as the Student's t log 
likelihood in place of the least squares penalty.
The goal is to develop formulations and estimators that achieve adequate performance when faced 
with outliers; these may arise either as gross measurement
errors, or real-world events that are not modeled by the dynamics.

Figure~\ref{GLT-KF} shows the probability density functions and  penalties
%and {\it influence functions} (see e.g.~\cite{Mar}) 
corresponding to Gaussian, Huber, and Student's t densities. 
Quadratic tail growth corresponds to extreme decay of the Gaussian density for large 
inputs, and linear growth of the influence of any measurement on the fit. In contrast, 
Huber and Student's t have linear and sublinear tail growth, respectively, 
which ensures every observation has bounded influence. 
%Student's t 
%has the additional feature that the influence of measurements
%with large residuals goes to $0$; this property is equivalent to the 
%sublinear growth of the penalty. 

We focus on the Huber function~\cite{Hub}, since it is both 
$C^1$-smooth and convex. In particular, the function $f(y)$ in~\eqref{eqn:main_stuff} and~\eqref{eqn:subprob_eff}
is chosen to be a composition of the Huber with an observation model.
Note that Huber is not $C^2$, so this case is not immediately captured by the theory 
we propose. However, Huber can be closely approximated by a $C^2$ function~\cite{bube1997hybrid}, 
and then the theory fully applies. 
For our numerical examples, we apply the algorithm developed in this paper directly to the 
Huber formulation.%, since we have an efficient interior point code for it. 

We illustrate robust modeling using a  simple representative example. 
Consider  a 2-dimensional oscillator, governed by the following equations: 
\begin{equation}
\label{eq:KalmanODE}
\begin{bmatrix} 
y_1\\
y_2
\end{bmatrix}'
= \begin{bmatrix}
-2u_1u_2 & -u_1^2 \\
1 & 0 
\end{bmatrix}
\begin{bmatrix} 
y_1\\
y_2
\end{bmatrix}
+
\begin{bmatrix} 
\sin(\omega t) \\
0
\end{bmatrix}
\end{equation}
where we can interpret $u_1$ as the frequency, and $u_2$ is the damping. 
Discretizing in time, we have 
\[
\begin{bmatrix} 
y_1\\
y_2
\end{bmatrix}^{k+1}
 = 
 \begin{bmatrix}
1 -2\Delta t u_1u_2 & - \Delta t u_1^2 \\
\Delta t & 1 
\end{bmatrix}
\begin{bmatrix} 
y_1\\
y_2
\end{bmatrix}^k + 
\Delta t\begin{bmatrix} 
\sin(\omega t_k)\\
0
\end{bmatrix}.
\]
%This problem has an unknown initial condition $\begin{bmatrix} 
%y_1^0\\
%y_2^0
%\end{bmatrix}$, 
%and unknown parameters $(u_1, u_2)$.
We now consider direct observations of the second component, 
\[
z_k = \begin{bmatrix} 0 & 1 \end{bmatrix}\begin{bmatrix} 
y_1\\
y_2
\end{bmatrix}^k  + w_k, \quad w_k \sim N(0, R_k).
\]

We can formulate the joint inference problem on states and measurements as follows: 
\begin{equation}
\label{eq:kalmanODE}
\min_{u,y}~ \phi(u,y) := \frac{1}{2}\rho\left(R^{-1/2}(H y - z)\right) \;\; \mbox{s.t.} \;\; Gy = v,
\end{equation}
with $v_k = \sin(\omega t_k)$, $v_1$ the initial condition for $y$, and $\rho$ is
either the least squares or Huber penalty, and we use the following definitions:
\[
\begin{aligned}
R       & =  \mathrm{diag} ( \{ R_k \} )
\\
H       & = \mathrm{diag} (\{H_k\} )
\end{aligned}\quad 
\begin{aligned}
y       & = \mathrm{vec} ( \{ y_k \} )
\\
z      & = \mathrm{vec} (\{z_1,  \dots, z_N\})\\
G  & = \begin{bmatrix}
    \mathrm{I}  & 0      &          &
    \\
    -G_2   & \mathrm{I}  & \ddots   &
    \\
        & \ddots &  \ddots  & 0
    \\
        &        &   -G_N  & \mathrm{I}
\end{bmatrix}
\end{aligned}
\]
Note in particular that there are only two unknown parameters, i.e. $u \in \mathbb{R}^2$, 
while the state $y$ lies in $\mathbb{R}^{2N}$, with $N$ the number of modeled time points. 

The reduced optimization problem for $u$ is given by 
\[
\min_u f(u) := \frac{1}{2}\rho(R^{-1/2} (HG(u)^{-1}v - z). 
\]
To compute the derivative of the ODE-constrained problem, we can use the adjoint state method. 
Defining the Lagrangian 
\[
\mathcal{L}(y,x,u) = \frac{1}{2}\rho\left(R^{-1/2}(Hy - z)\right) + \langle x, Gy - v\rangle,
\]
we write down the optimality conditions $\nabla \mathcal{L} = 0$ and obtain
\[\left\{
\begin{aligned}
\overline y &= G^{-1}v \\
\overline x & = -G^{-T}(H^TR^{-1/2}\nabla\rho(R^{-1/2}HG^{-1}v - z)) \\
\nabla_u f  & = \left\langle \overline x, \frac{\partial (G(u) \overline y)}{\partial u}\right\rangle
\end{aligned}\right\}.
\]
The inexact (penalized) problem is given by 
\begin{equation}
\label{eq:kalman_pen}
\min_{u,y}~ \varphi(u,y) = \frac{1}{2}\rho\left(R^{-1/2}(H y - z)\right) + \frac{\lambda}{2} \|Gy - v\|^2.
\end{equation}
We then immediately find
\[
\begin{aligned}
\overline y & = \arg\min_{y}\frac{1}{2}\rho\left(R^{-1/2}(H y - z)\right) + \frac{\lambda}{2} \|Gy - v\|^2\\
\nabla   \widetilde \phi(u)& = \lambda \frac{\partial (G(u)\overline y)}{\partial u} (G\overline y - v).
\end{aligned}
\]
When $\rho$ is the least squares penalty, $\overline y$ is available in closed form. 
However, when $\rho$ is the Huber, $\overline y$ requires an iterative algorithm. 
Rather than solving for $\overline y$ using a first-order method, we use IPsolve, an interior point 
method well suited for Huber~\cite{JMLR:v14:aravkin13a}. 
Even though each iteration requires inversions of systems of size $\mathcal{O}(N)$, 
these systems are very sparse, and the complexity of each iteration to compute $\overline y$
is $\mathcal{O}(N)$ for any piecewise linear quadratic function~\cite{JMLR:v14:aravkin13a}. Once again, we see that the computational cost does not scale with $\lambda$.

\subsubsection{Numerical Experiments} We simulate a data contamination scenario by solving 
the ODE~\eqref{eq:KalmanODE} for the particular parameter value $u = (2, 0.1)$.
The second component of the resulting state $y$ is observed, and the observations are contaminated. 
In particular, in addition to Gaussian noise with standard deviation $\sigma = 0.1$, 
for 10\% of the measurements uniformly distributed errors in $[0,2]$ are added. 
The state $y \in \mathbb{R}^{2(4000)}$ is finely sampled over 40 periods. 

For the least squares and the Huber penalty with $\kappa = 0.1$,
we solved both the ODE constrained problem~\eqref{eq:kalmanODE}
and the penalized version~\eqref{eq:kalman_pen} for $\lambda \in \{10^3, 10^5, 10^7, 10^9\}$. 
The results are presented in Table~\ref{tab:results}. The Huber formulation behaves analogously to 
the formulation using least squares; in particular the outer (projected) function in $u$ is no more difficult to minimize.
And, as expected, the robust Huber penalty finds the correct values for the parameters. 

A state estimate generated from $u$-estimates corresponding to large $\lambda$ is shown in Figure~\ref{fig:kalman_1}. 
The huberized approach is able to ignore the outliers, and recover both better estimates 
of the underlying dynamics parameters $u$, and the true observed and hidden components 
of the state $y$.

\begin{table}
\caption{Results for Kalman experiment. Penalty method for both least squares and huber
achieves the same results for moderate values of $\lambda$ as does the projected 
formulation. While Huber results converge to nearly the true parameters $u$, 
least squares results converge to an incorrect parameter estimate. \label{tab:results}}
\centering
\begin{tabular}{|c|c|c|c|c|}\hline
$\lambda$ & $\rho$ & Iter & Opt & $u$\\\hline
$10^3$ & $\ell_2$ & 9 & $3.1\times 10^{-7}$ & $(.45, .98)$\\ \hline 
$10^5$ & $\ell_2$ & 18 & $3.2\times 10^{-7}$ & $(.15, 4.3)$\\ \hline 
$10^7$ & $\ell_2$ & 26 & $5\times 10^{-5}$ & $(.07, 11.1)$\\ \hline 
$10^9$ & $\ell_2$ & 31 & $4\times 10^{-6}$ & $(.07, 11.8)$\\ \hline 
$\infty$ & $\ell_2$ & 29 & $3.3\times 10^{-7}$ & $(.07, 11.8)$\\ \hline
&&&&\\\hline 
$10^3$ & h & 9 & $2.4\times 10^{-7}$ & $(1.92 .14)$\\ \hline 
$10^5$ & h& 12 & $5\times 10^{-6}$ & $(1.98, .11)$\\ \hline 
$10^7$ & h & 10 & $5\times 10^{-6}$ & $(1.99, .11)$\\ \hline 
$10^9$ & h & 13 & $1\times 10^{-5}$ & $(1.99, .11)$\\ \hline 
$\infty$ & h & 17 & $2\times 10^{-6}$ & $(1.99, .11)$\\ \hline
\end{tabular}
\end{table}

\begin{figure}[h]
%\hspace{-.4in}
\centering
	\includegraphics[scale=.4]{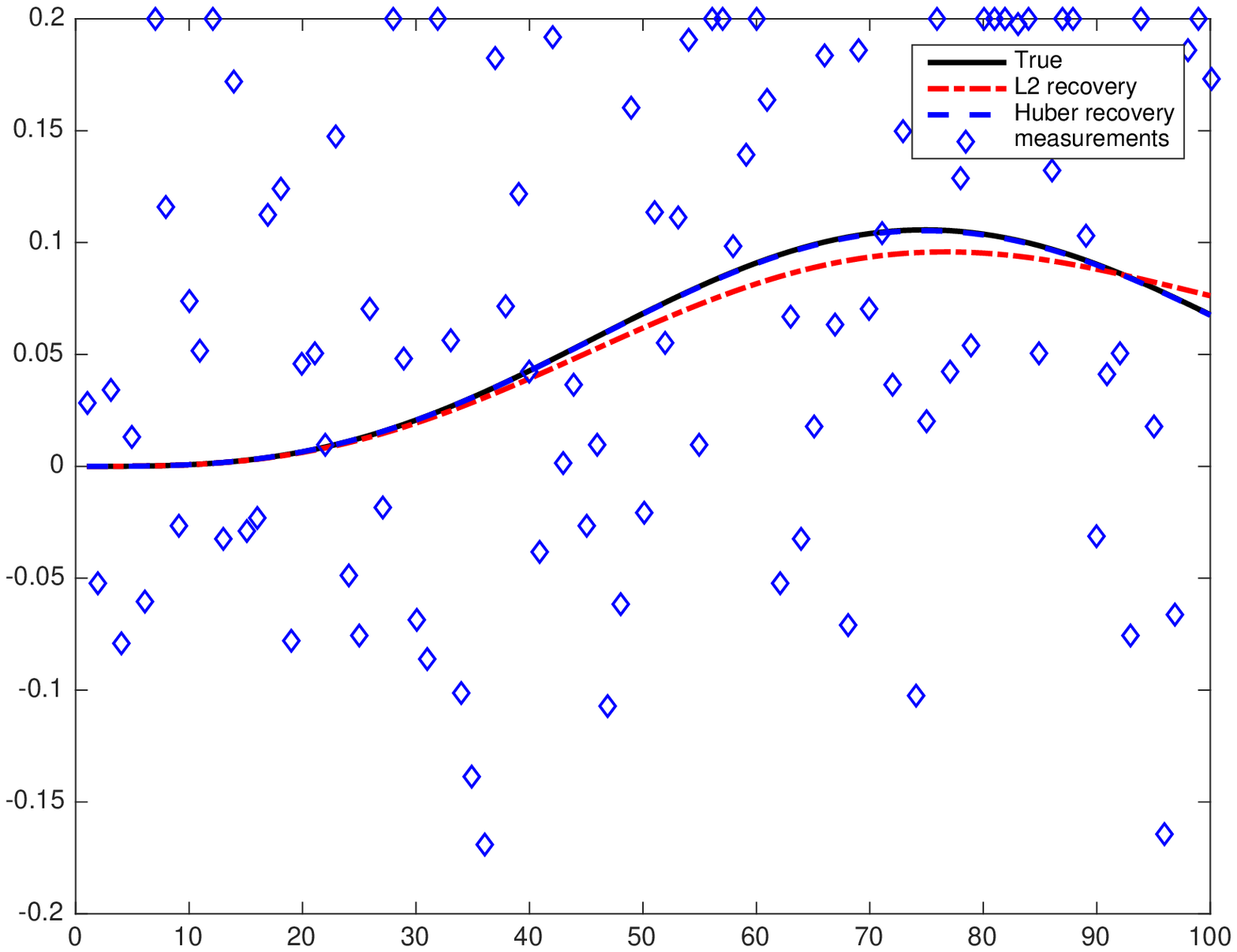}
	\includegraphics[scale=.4]{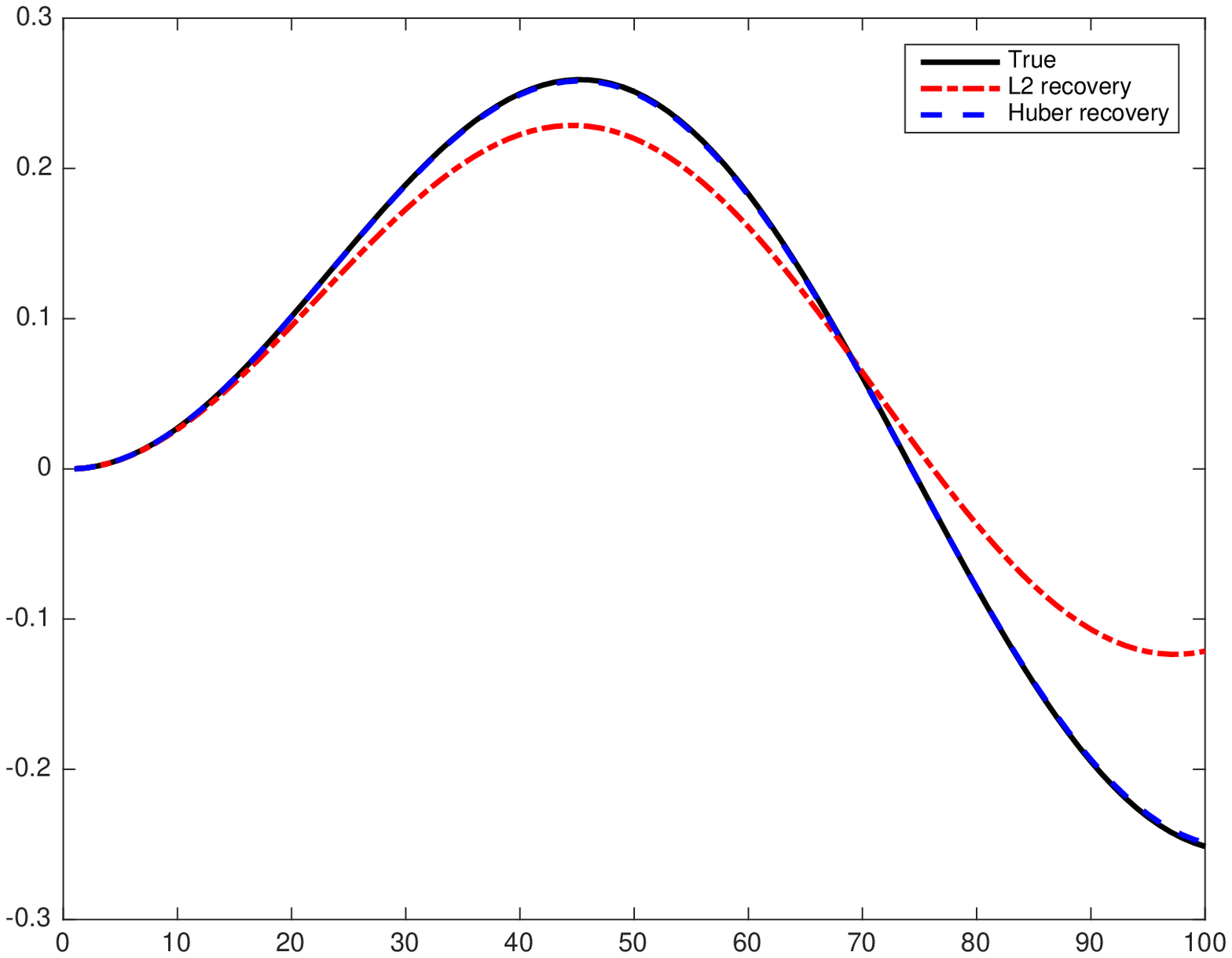}
	\caption{Top panel: first 100 samples of state $y_2$ (solid black), least squares recovery (red dash-dot) and huber recovery (blue dash). Noisy measurements $z$ appear as blue diamonds on the plot, with outliers shown at the top of the panel. Bottom panel: first 100 samples of state $y_1$ (solid black), least squares recovery (red dash-dot) and Huber recovery (blue dash).  }
	\label{fig:kalman_1}
\end{figure}

%\subsubsection{Robust misfit measures}
%
%We can extend the entire framework to use for example huber applied to misfit measures. 
%
%In particular, we can extend it to scaled and unscaled Huber. 
%Scaled huber that converges to the 1-norm as $\kappa \downarrow 0$ is given by 
%\[
%\rho_s(r) = \sup_{u\in [-1,1]} ur - \frac{\kappa}{2}u^2 = 
%\begin{cases}
%\frac{1}{2\kappa} r^2 & \mbox{if} \quad |r| \leq \kappa \\
%|x| - \frac{\kappa}{2} & \mbox{if}\quad |r| > \kappa.
%\end{cases}
%\]
%while the unscaled version that converges to the quadratic for $\kappa \uparrow \infty$ 
%is given by 
%\[
%\rho(r) = \sup_{u\in [-\kappa,\kappa]} ur - \frac{1}{2}u^2 = 
%\begin{cases}
%\frac{1}{2} r^2 & \mbox{if} \quad |r| \leq \kappa \\
%\kappa|x| - \frac{\kappa^2}{2} & \mbox{if}\quad |r| > \kappa.
%\end{cases}
%\]

\section{Conclusions\label{sec:conclusions}}
In this paper, we showed that, contrary to conventional wisdom, the quadratic penalty technique 
can be used effectively for control and PDE constrained optimization problems, if done correctly. 
In particular, when combined with the partial minimization technique, we showed that the penalized projected scheme 
\[
\min_{u}~ g(u) + \min_{y} \left\{ f(y) + \frac{\lambda}{2}\|A(u)y - b\|^2\right\}
\]
has the following advantages: 
\begin{enumerate}
\item The Lipschitz constant of the gradient of the outer function in $u$ is bounded as $\lambda \uparrow \infty$, and hence we can effectively analyze the global convergence of first-order methods. 
\item Convergence behavior of the data-regularized convex inner problem is controlled by parametric matrix
$A(\cdot)$, a fundamental quantity that does not depend on $\lambda$. 
\item The inner problem can be solved inexactly, and in this case, the number of inner iterations (of a first-order algorithm) needs to grow only logarithmically with $\lambda$ and the outer iterations counter, to preserve the natural rate of gradient descent. 
\end{enumerate}

As an immediate application, we extended the penalty method in~\cite{van2015penalty} to convex robust formulations, using the Huber penalty composed with a linear model as the function $f(y)$. Numerical results illustrated the overall approach, including convergence behavior of the penalized projected scheme, as well as modeling advantages of robust penalized formulations. 

\bigskip

\noindent{\bf  Acknowledgment.}
Research of A. Aravkin was partially supported by the Washington Research Foundation Data Science Professorship. 
Research of D. Drusvyatskiy was partially supported by the AFOSR YIP award FA9550-15-1-0237. The research of T. van Leeuwen was in part financially supported by the   Netherlands Organisation of Scientific Research (NWO) as part of research programme 613.009.032.

\bibliographystyle{abbrv}
\bibliography{mybib,references_sasha}

\end{document}